\documentclass[10pt]{amsart}
\usepackage[english]{babel}
\usepackage{latexsym,amsthm,amssymb,amsmath,amsfonts,euscript,amstext}
\usepackage[T1]{fontenc}
\usepackage[utf8]{inputenc}
\usepackage{epsfig}
\usepackage{graphicx,epsfig,color,epstopdf,float}
\usepackage{hyperref}

\newtheorem{theorem}{Theorem}[section]
\newtheorem{corollary}{Corollary}

\newtheorem{proposition}{Proposition}
\newtheorem{conjecture}{Conjecture}

\theoremstyle{definition}
\newtheorem{definition}[theorem]{Definition}
\newtheorem{remark}{Remark}

\title{Boundaries  of\\
 Siegel Disks for 
Conservative Systems}

\author{F.M. Tangerman}

\begin{document}
\maketitle

\date{today}


\tableofcontents

\section{Introduction} \label{chap:intro}
\subsection{Motivation}
Real analytic conservative maps on the plane with a real analytic invariant circle of sufficiently irrational rotation number have also, by standard KAM theory, an invariant annulus, a Herman ring,  in their natural complexification, see also \cite{BS}. Such a Herman ring is foliated by invariant curves that are invariant under the dynamics. Invariant curves in the open Herman ring are
real analytic. We are interested in the regularity of the boundaries
of such a Herman ring. There seems to be little general theory about the structure of the boundary or ends of such a Herman ring, not in the least because these are hard to find by computational analytic means. Here are a few of the questions:
\begin{enumerate}
\item Are the boundary components of a Herman ring topological circles? We suspect that even that may depend on the rotation number. 
\item Is the dynamics topologically transitive on each of the boundary components?
\item If the boundary components are circles, what is their regularity? On general principles these circles should not be analytic: if they were, presumably KAM theory would be able to extend these annuli beyond these curves. Thus these annuli are also at their 'natural' limit of analyticity, for instance defined using the concept of polynomial hull.
\end{enumerate}

When we consider deformations of the dynamics and track Herman rings with
constant rotation number, it is possible for such a Herman ring to reduce to a single curve, the 'last invariant' curve. When the rotation is sufficiently irrational a 'last invariant' curve can not be real analytic, and can not be even $C^2$, as that would automatically imply real analytic. For 'noble' rotation numbers the universality properties of such 'last invariant' curves have been thoroughly, but not entirely conclusively, investigated in the conservative context by the work of R. McKay, see \cite{M,M2}. The general regularity of 'noble' last invariant curves is believed to be around $C^{1.7}$.
It is not even clear if the dynamics on such curves has
to be topologically transitive (i,e no wandering intervals, i.e. no Denjoy counter examples). It is clear, via Aubry Mather theory, that last invariant curves disintegrate into hyperbolic Cantor sets, (Cantori), of prescribed rotation number.
 
In the dissipative context such invariant curves are normally hyperbolic and have been studied more conclusively by D. Rand  et al \cite{ROSS,R}.  In that case such curves can not be even $C^1$.
We refer to \cite{BST} for a general discussion of the bifurcation mechanisms, compelling imagery and appendices, associated with the destruction of invariant circles in dissipative systems.

We contrast this with the study of Herman rings and Siegel disks
for holomorphic or meromorphic maps on the complex plane $\mathbb{C}$ where the understanding of regularity and transitivity properties of the boundaries of Siegel disks and Herman rings is much more developed. Rather than providing specific references we mention
here the lead actors:  Herman, Yoccoz, Jones, Swiatek, Graczyk, Shishikura, Yampolski, Buff, Cheritat, and Avila. For rotation numbers of bounded type the boundary of the Siegel disk tends to
be a fractal containing a critical point. This is certainly the case when the dynamics extends
holomorphically to a neighborhood of the Siegel disk. For entire functions this does not need to be the case and boundaries can be smooth (depending on the nature of the essential singularity). In the $\mathbb{C}^2$ there is little in the way of 'conformal background' although there is a complex structure. One would therefore anticipate that maps have a natural extension only to the closure of domains of holomorphy and not beyond. In that respect there \textbf{should} be a a higher degree of similarity in the behavior of entire functions and that found in conservative systems. 

The primary difficulty in the analysis of the such 'last invariant' curves is that they are difficult to find. We believe that boundaries of Herman rings in the conservative setting may be the natural candidate to study first, when they can be found.

\subsection{Siegel Disks}
   
We consider a specific example where these Herman rings are easily studied, since they \underline{degenerate} into Siegel disks. These examples are related to the interior solution of the KAM problem. Specifically consider the following conservative map introduced in \cite{GP}:
\begin{equation}
(q,p)\to (q+p+ie^{iq},p+ie^{iq})
\label{equation:inner solution}
\end{equation}
where $q$ is defined on the holomorphic cylinder $\mathbb{C}/{2\pi\mathbb{Z}}$ and $p$ is defined on
$\mathbb{C}$. When $q$ tends to infinity in the \textbf{positive} imaginary direction
this map limits to the integrable case:
\begin{equation}
(q,p)\to (q+p,p)
\end{equation}
that has for every real $p$ invariant curves.
Specifically, let $Z=e^{iq}$, $W=e^{ip}$, then Equation \ref{equation:inner solution} reduces to:
\begin{equation}
(Z,W)\,\to (ZWe^{-Z},We^{-Z})
\label{equation:inner solution Siegel}
\end{equation}
which has the curve $Z=0$ as set of fixed points. When $|Z|$ is small
the dynamics is approximately $(Z,W)\to (ZW,W)$ and with $W=\lambda=e^{2\pi i \alpha}$ a rotation by angle $\alpha$.

It is a remarkable property of this map is that for any given $\lambda$ on the unit circle, the corresponding invariant Siegel disk on which the dynamics is conjugate to the holomorphic map $x\to\lambda x$ is easily described in terms
of a power series with \emph{positive} coefficients that is furthermore \emph{convergent} on its radius of convergence.

Precisely, in \cite{SV}, which also concentrates on the domain of analyticity
for KAM rings  or Herman rings in conservative systems it was observed that the dynamics on the Siegel disk:
\begin{equation}
(Z(x),W(x))\to (Z(\lambda x),W(\lambda x))
\end{equation}
can be described as follows.
The point $(0,\lambda)$ is the center of the Sigel disk and we factor
$Z(x)$ as $Z(x)=xe^{m(x)}$, with $Z(0)=0$. We may assume the normalization $Z'(0)=1$, i.e. $m(0)=0$. Since $W(\lambda x) = \dfrac{Z(\lambda x)}{Z(x)}=\lambda e^{m(\lambda x)-m(x)}$ one obtains as equation for $m$:
\begin{equation}    
m(\lambda x)-2m(x)+m(\dfrac{x}{\lambda})=-xe^{m(x)}
\label{equation: standard like map}
\end{equation}
which has these properties:
\begin{itemize}
\item When the rotation number $\alpha$ is sufficiently irrational, the power series solution $m(x)=\sum_k m_k x^k$ has \underline{positive}
coefficients and a \underline{positive} radius of convergence.
\item On any circle $|x| = r$ we have the apriori bound for $M(r)= max_{|x|=r}|m(x)|$   
$$re^{M(r)}\,\leq\, 4 M(r)$$or equivalently:
\begin{equation}
r\,\leq\,4 M(r)e^{-M(r)}
\end{equation}
The function $f(M)=Me^{-M}$ $M>0$ has a maximum when $M=1$, and tends to zero
as $M$ tends to infinity. Thus the radius of convergence $r_{\alpha}$ is bounded by $\dfrac{4}{e}$ and $M$ remains bounded as $r$ tends  to $r_{\alpha}$, since otherwise $r$ would have to tend to $0$. Since the coefficients are positive, $m(x)$ 
is then \underline{bounded and continuous} on the radius of convergence $r_{\alpha}$.
\end{itemize} 
We consider the \emph{embedding} of the Siegel disk in $\mathbb{C}^2$:
\begin{equation}
x\to \Gamma(x)=(Z(x), m(\lambda x)-m(x))
\end{equation}
This embedding extends to the closed disk $|x|\,\leq\,r_{\alpha}$ as a continuous
function. We have an immediate conclusion: the dynamics on the boundary
of its image is topologically transitive: 
there are no wandering intervals, because the map $\theta\,\to \Gamma(r_{\alpha} e^{i \theta})$ is a continuous semi-conjugacy.
However it is not immediate that this map is a conjugacy.

We propose the following conjecture:

\begin{conjecture}
If $\alpha$ is of bounded type then the map
$\theta \to \Gamma(r_{\alpha} e^{i \theta})=\Gamma_{\alpha}(\theta)$ 
is at least $C^1$.
\end{conjecture}

If we can prove this conjecture then we have an immediate important corollary:

\begin{corollary}
If $\alpha$ is of bounded type then the map
$\theta \to \Gamma_{\alpha}(\theta)$ is an embedding and hence a conjugacy.
\end{corollary}

\begin{proof}
Since the map $\Gamma_{\alpha}$ is $C^1$ we can consider its singularities,
the points $\theta$ where $\dfrac{d}{d \theta} \Gamma_{\alpha}$ vanishes.
It follows easily from differentiation of Equation \ref{equation: standard like map} that the set of singularities is invariant under the map.
Since the derivative is $C^0$ there is then an alternative:
\begin{itemize}
\item The derivative is nowhere zero, in which case the map, $ \Gamma_{\alpha}$ is an immersion, or,
\item The map $\Gamma_{\alpha}$ is constant. 
\end{itemize}
The last case is easily ruled out (as we will see).
The points where $\Gamma_{\alpha}$ fails to be an embedding is also seen to
be invariant under the dynamics, and therefore must be empty.
\end{proof}

We then conjecture, and show numerical evidence, that the image curve, the boundary of the Siegel disk,
 is not particularly smooth:

\vskip.1in
\noindent

\begin{conjecture}
If $\alpha$ is of bounded type then there is a $1<p<3$ so that the map
$\theta \to \Gamma_{\alpha}(\theta)$ 
is not of class $C^p$.
\end{conjecture}

\begin{remark}
We conjecture that a sharp estimate for $p$ would be a bit higher than the McKay's (\cite{M}) estimate of around $1.7$ for the 'last' noble invariant circle. There appears to be more 'wiggle room' in $\mathbb{C}^2$.
\end{remark}

\vskip.1in
\noindent

We next consider the embedding properties of the closure of the Siegel disk.
Since it is embedded in $\mathbb{C}^2$ it could still be complicated, but
we show it is in fact a graph. Clearly the map $x\to Z(x)$ is univalent near $x=0$. We then have that the correspondence:
$Z(x)\to x \to W(x)$ defines a holomorphic map $Z\,\to\,W\equiv H(Z)$ and we show that this extends.

\begin{figure}
\centering
{\includegraphics[width=3.0in]{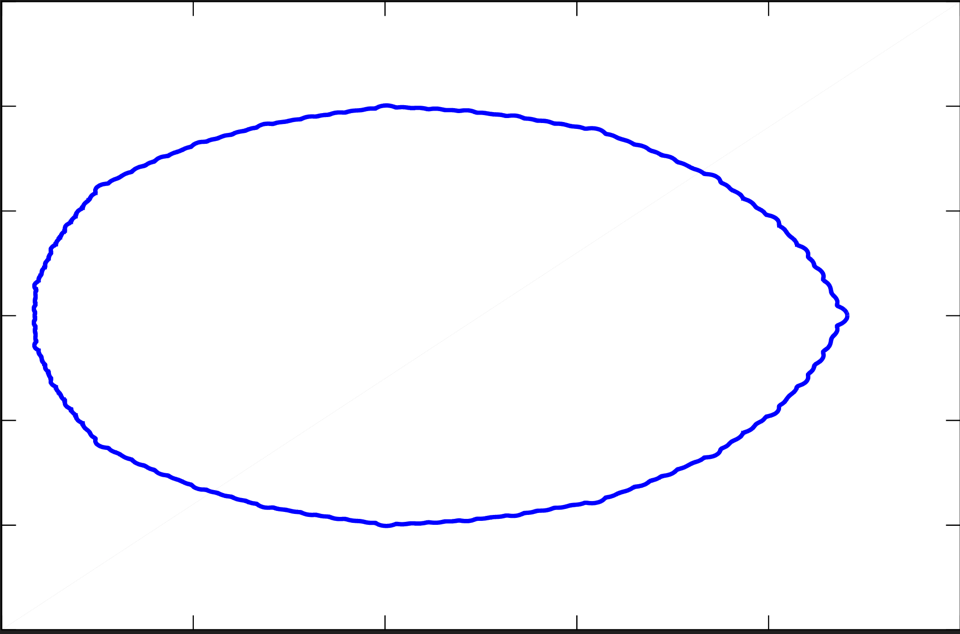}} \\
\caption{
\emph{ 
Shown is the image of a circle close to the radius of convergence.
It is more or less evident that there are no self-intersections, and
that the map $x\to Z(x)$ is univalent. It also shows that the boundary
has small detail wiggles on, suggesting that it may not be that smooth.
 }}
\label{fig:Z-image}
\end{figure}

\begin{figure}[ptbh]
\centering
\scalebox{0.5}{\includegraphics{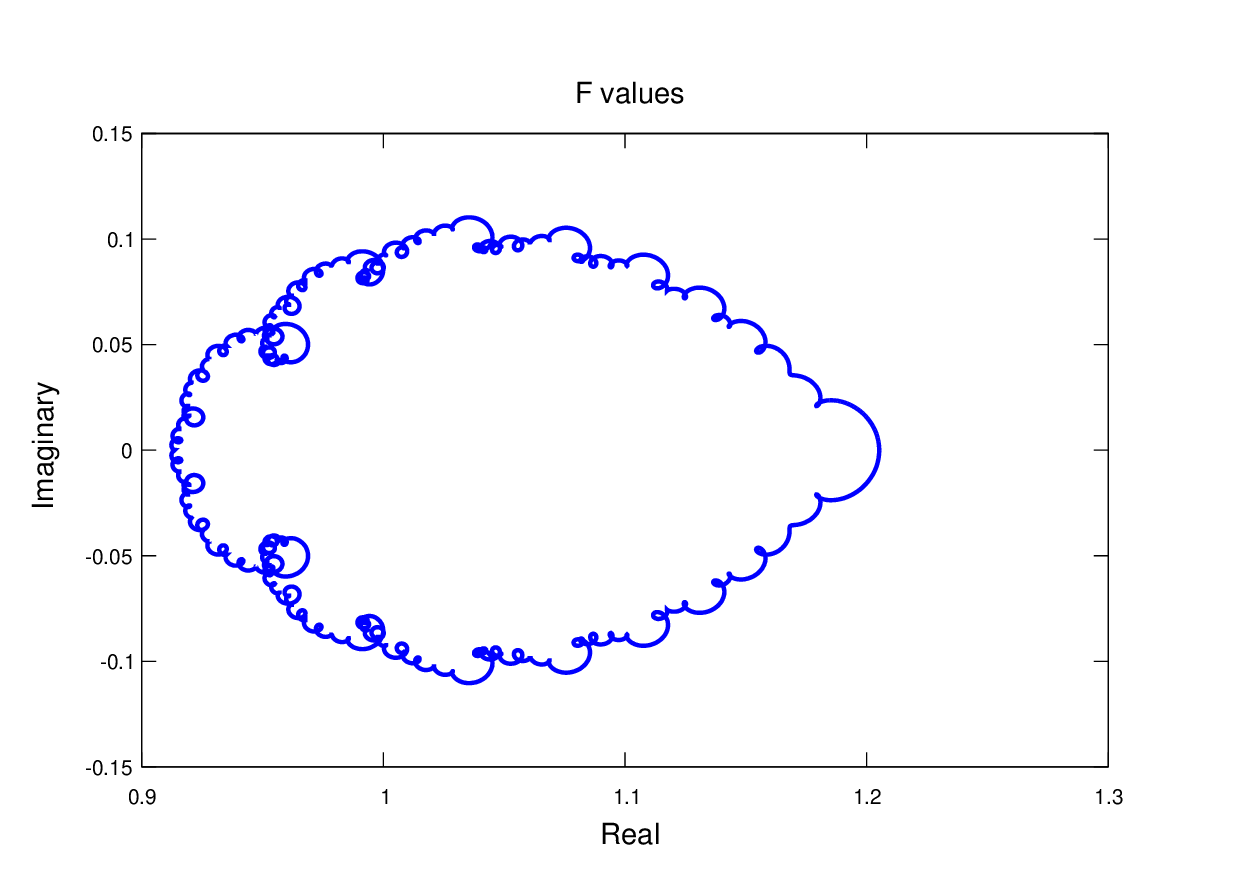}} \\
\caption{\emph{ 
Shown is the phase of the map $x\to Z(x)=xF(x)$ at a circle close to the radius of convergence. The map $x\to F(x)$ is clearly not univalent
and has a lot of fine structure particularly toward the left of the
image.
 }}
\label{fig:F-image}
\end{figure}

As part of this project we propose as conjecture:

\begin{conjecture}
If $\alpha$ is of bounded type then the corresponding Siegel disk is a graph: its image is of the form $W=H(Z)$ with $H$ holomorphic.
\begin{itemize}
\item
The map $Z(x)\to Z(\lambda x)$ is given by a holomorphic map
$Z\to G(Z)$ defined on the projection of the Siegel disk in the $Z-plane$ and:
\begin{equation}
G(Z)G^{-1}(Z)=Z^2e^{-Z}
\label{eq: G equation}
\end{equation}
\item
$G$ extends to the boundary as a diffeomorphism of class $C^{<p}$.
\end{itemize}
\end{conjecture}

Equation \ref{eq: G equation} is a holomorphic version of the approach introduced by Mather for the standard map: to determine parameter values for which there are no invariant circles:
\begin{equation}
g(x) + g^{-1}(x) = 2 x + \dfrac{k}{2\pi}sin(2\pi x) 
\end{equation}


\section{Master Equation}

The key idea, introduced in \cite{SV}, is to study the recursion
relation for the coefficients of the function $m(x)=\sum_k m_k x^k$.

The recursion relation for the coefficients is quite simple,
see \cite{SV} Equations 3.4 and 3.5 \footnote{
We are differing by a factor of 2 due to the slightly form of the
equation, easily absorbed by a linear change of 'x' variable.}.
Let $F(x)=e^{m(x)}$. With $m(0)=0$ (normalization condition)
we have that $F(x)=\sum_k F_k x^k$ with $F_0=1$ and

\begin{eqnarray}
F(x)=e^{m(x) }\Longleftrightarrow kF_k &=&\sum_{j=1}^{k} jm_j F_{k-j}\\
m(\lambda x) - 2m(x) + m(\dfrac{x}{\lambda})=-xF(x) \Longleftrightarrow m_{k+1} &=& \dfrac{F_k}{{{D}}(k+1)}
\label{eq:recursion}
\end{eqnarray}

where
${{D}}(k)\equiv 2(1-cos(k\alpha))= 4sin^2(k\dfrac{\alpha}{2}) >0$.
The reason we leave the factor of $2$ in the definition of
${{D}}$ is two-fold: $2(1-cos(x))\sim x^2$ and the function $\ell(x)\equiv\ln(2(1-cos(x))$ is the periodic analog of the function $\ln(x)$ for the Fourier transform.

From Equation \ref{eq:recursion} it is clear the coefficients
of $m$ and those of $F$ are positive.
Let $\alpha$ be sufficiently irrational. We can assume that
the radius of convergence $r\equiv r_{\alpha}$ is positive.
We then conclude (as in the introduction) that the functions
$m(x)$ and therefore also $F(x)$ are 
uniformly convergent when $|x|=r$:
\begin{equation}
\sum_k F_k r^k <\infty
\end{equation} 

Express Equation \ref{eq:recursion} in a single Master equation
for the coefficients $G(k,r)\equiv F_k r^k$ and $G(0,r)=1$ and $r\leq r_{\alpha}$:
\begin{equation}
Master:\;k G(k,r) = r\sum_{j=1}^k j \dfrac{G(j-1,r)}{{{D}}(j)}G(k-,rj)
\label{eq:master equation}
\end{equation}

This type of relation is very common in non-linear systems and
is by itself not very promising without some additional information:
the coefficients $G(k)$ are positive and
\begin{equation}
\sum G(k,r) < \infty
\end{equation}

\textbf{Note:}
In 'practice' the radius of convergence can be determined by the following
equivalent recipes. For any given $r$ determine the sequence $G(k)(r)$
from the Master Equation and determine if it tends to infinity or not:
\begin{equation}
r_{\alpha} \equiv \sup_r \{\limsup_k G(k,r) < \infty \}
\end{equation}
Since the equation has positive coefficients the equivalent criterion
is:
 \begin{equation}
r_{\alpha} \equiv \inf_r \{\limsup_k G(k,r) = \infty \}
\end{equation}

\begin{figure}[ptbh]
\centering
\scalebox{0.5}{\includegraphics{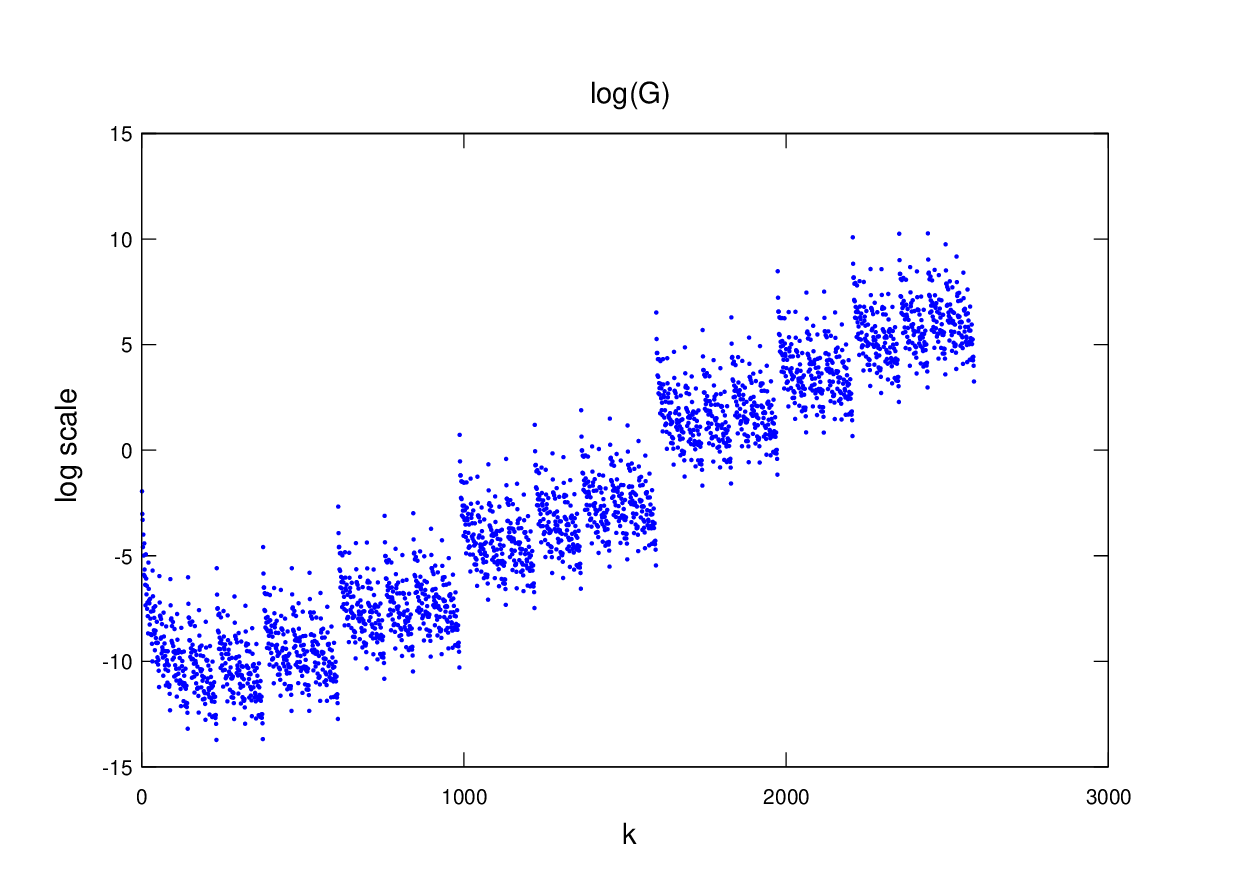}}
\caption{\emph{ 
Shown is the sequence $\ln(G(k))$ with radius $r$ slightly above
the radius of convergence, for the golden mean rotation number.
The linear growth indicates that we are above the radius of convergence.
There is however 'fine' detail that needs to also be explained.
 }}
\label{fig:logG-image}
\end{figure}

As is evident from Figure \ref{fig:logG-image} there is some pattern, that we will
elucidate first on the 'first order' approximation of the Master Equation (only
the first term):
\begin{equation}
Master\;First\;Order:\;G_1(k,r) = r\dfrac{G_1(k-1,r)}{{{D}}(k)}
\label{eq:first master equation}
\end{equation}

We have as the coefficients are positive that for any $r$, $G(k,r) > G_1(k,r)$.
We also have an explicit solution for $G_1(k,r)$:
\begin{equation}
G_1(k,r)=r^k\dfrac{1}{\prod_1^k {{D}}(j)}
\end{equation}

Note that the first order Master Equation corresponds to solving the "linearized" functional  equation, with $m(0)=1$, (write $e^m=1+m$ and replace $m$ by $1+m$):
\begin{equation}
m(\lambda x)-2m(x)+m(\dfrac{x}{\lambda})=-xm(x)
\label{eq: linearized equation}
\end{equation}

We will study its solution in the next section.

\vskip.2in

\section{Linearized Master Equation} \label{section :linearized master}

\vskip.2in
The radius of convergence of the linearized Master Equation is considerably
larger than that for the system of interest.
First we have the upper bound, as before, with $M(r)=max_{|x|=r}|m(x)|$:
$$
rM(r)\,\leq 4 M(r)
$$
and thus $r\,\leq\,4$. 
We note that this estimate provides no implicit bound on $M$ (linear equation)
and that will show up in 'poor' control on the radius of convergence.

We will show that the radius is exactly equal
to one when $\lambda$ is irrational.
What is required is to determine:
$$
\lim_{n\to\infty} \dfrac{\ln (\prod_{j=1}^n {{D}}(j))}{n}=
\lim_{n\to\infty}\dfrac{1}{n}\sum_{j=1}^n \ln({{D}}(j))=\dfrac{1}{2\pi} \int_0^{2\pi} \ell(x)\,dx
$$
using the ergodic theorem\footnote{
The function $\ell(x)$ has only a logarithmic singularity and is therefore integrable (in $L^1$ and in $L^2$). Therefore the ergodic theorem can be applied. Its derivative does have a singularity to
be conquered.}.

We now have the following simple identities for the Fourier Transform:
\begin{eqnarray}
\ell(x)=\ln(2(1-cos(x))&=&\ln((1-e^{ix})(1-e^{-ix})\\
               &=&\ln(1-e^{ix})+\ln(1-e^{-ix})\\
               &=&-\sum_{k=1}^{\infty} \dfrac{e^{ikx}+e^{-ikx}}{k}
\label{eq: fft ln}
\end{eqnarray}
Since the constant term is 'absent' it must be that:
$\int_0^{2\pi} \ln(2(1-cos(x)) dx = 0$.

Therefore:

\begin{proposition}
Let $\lambda$ be irrational and 
let $m(x)$ be the solution of the linearized Master Equation, with $m(0)=1$
then the power series for $m$ has radius of convergence equal to 1.
\end{proposition}

\begin{proof} Consider
the set of $r\geq\,0$ for which:
\begin{eqnarray}
lim_{k\to\infty} \dfrac{1}{k} ln(\dfrac{r^k}{\prod_{j=1}^n {{D}}(j)})\,&\leq&\,0\\
\ln(r)-lim_{k\to\infty} \dfrac{1}{k} \sum_{j=1}^n \ln({{D}}(j))\,&\leq&\,0
\end{eqnarray}
Since the second term is equal to zero this reduces
to the relation: $\ln(r)\,\leq\,0$, proving that the radius of convergence
is equal to 1.
\end{proof}

The power series for the solution $m$ of the linear Master Equation
has positive coefficients and one might ask
if it is ever uniformly convergent on the boundary of the unit disk, the circle of radius equal to 1.
The answer is simply no.
Convergence is always subtler: while the 
logarithm of the radius of convergence in this example
is related to the \textbf{mean} of the sequence $\ln({{D}}(n))$, 
the \textbf{deviations from the mean} determine actual convergence at the radius of convergence. In this example we will have small denominators
and as a result one expects that this series never converges at the radius
of convergence. We will demonstrate this now. Consider the series:
$$
\sum_n \dfrac{1}{\prod_{j=1}^n {{D}}(j)}
$$
If this were convergent then 
$$
\lim_{n\to\infty} {\prod_{j=1}^n {{D}}(j)} = \infty
$$
Consider therefore the following sequence of Weyl sums:

$$
S\ell(n)\equiv\ln \prod_{j=1}^n {{D}}(j)\,=\,\sum_{j=1}^n \ln({{D}}(j))
$$
Convergence  thus requires that:
$$
\lim_{n\to\infty} S\ell(n) = \infty
$$

Figure \ref{fig:SlD-image} shows the distribution properties of the
sequence $S\ell(n)$, and that sequence does not appear to go infinity.

\begin{figure}[ptbh]
\centering
\scalebox{0.3}{\includegraphics{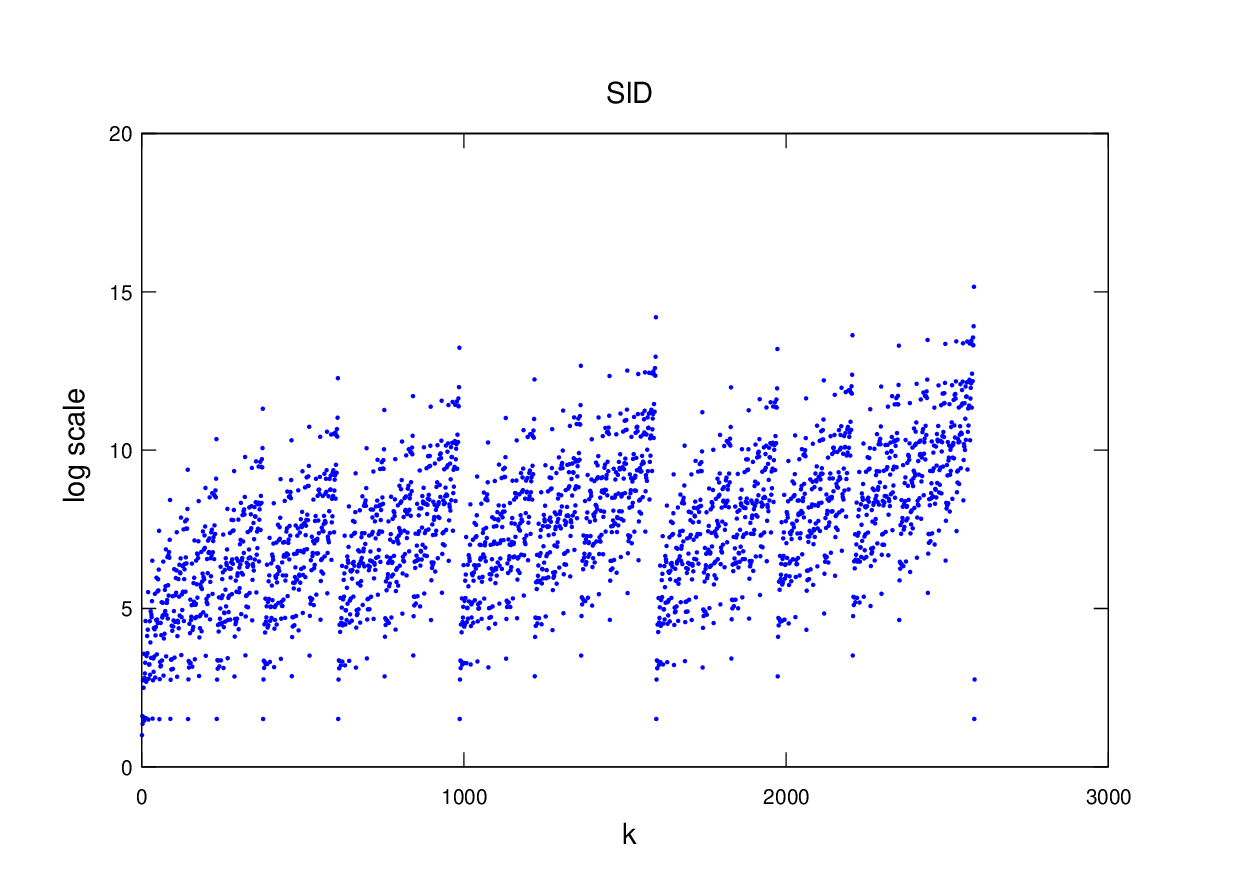}} \\
\caption{\emph{ 
Shown is the sequence $S\ell(k)$, for the same index range.
This sequence has two envelopes: a lower envelope which
is approximately constant,
and an upper envelope that is approximately linear in $\log(k)$.
The lower envelope corresponds to the values $S\ell(q_n)$, while the upper
envelope corresponds to the indices just preceding: $S\ell(q_n-1)$.
 }}
\label{fig:SlD-image}
\end{figure}

\begin{figure}
\centering
\scalebox{0.3}{\includegraphics{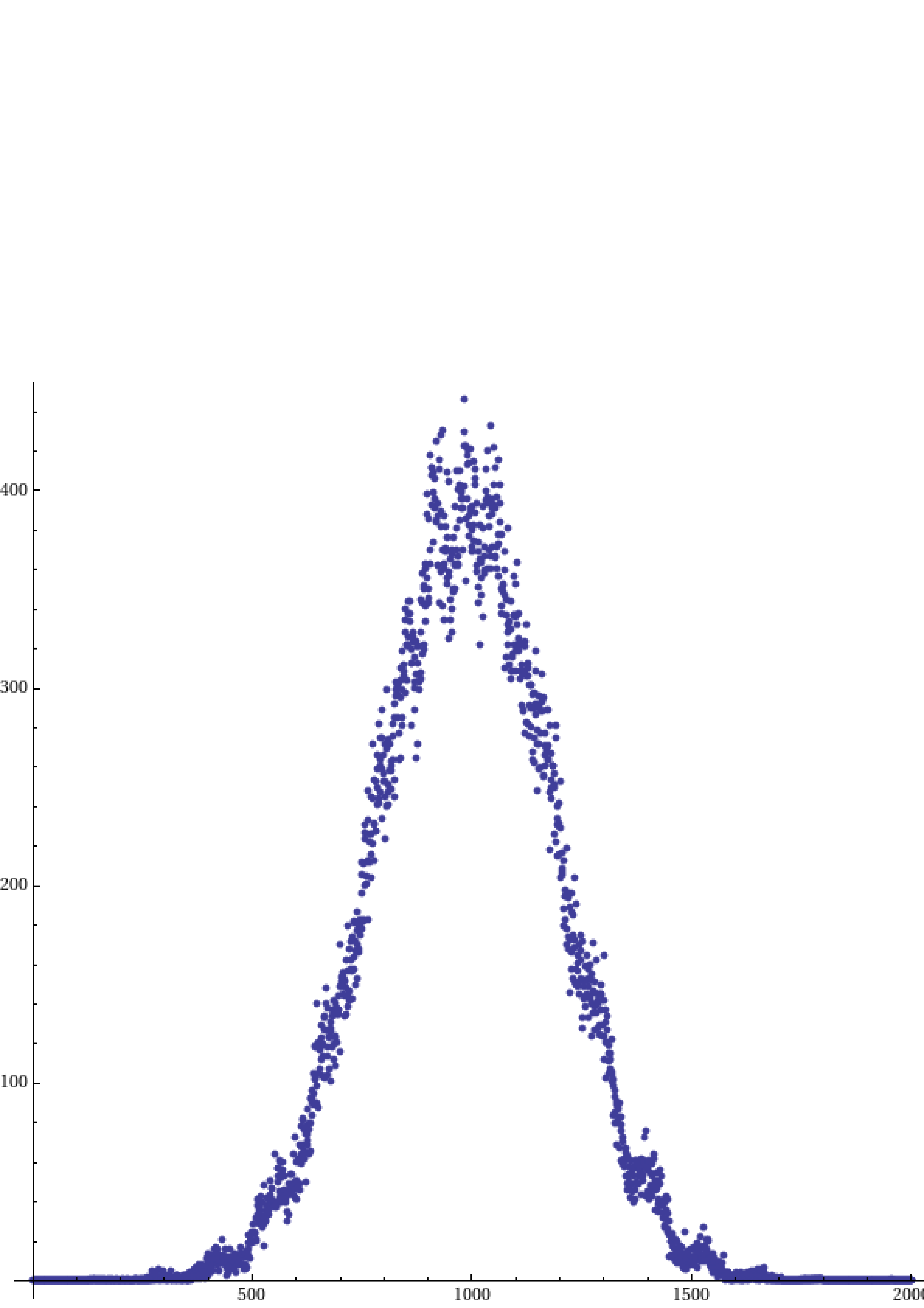}} \\
\caption{The distribution of $S_k(\alpha)/log(k)$, where
$k$ ranges from $1,...,q_{27}$, sampled at 2000 uniform
bins covering the interval $[0,2]$. This distribution
has \textbf{compact support} and looks symmetric but complicated.
}
\label{distribution}
\end{figure}

\begin{theorem}[Distribution Properties]
With $\alpha$ equal to the golden mean
\begin{enumerate} 
\item 
\begin{eqnarray}
\limsup_n S\ell(n)/\log(n)&=&2\\
\liminf_n S\ell(n)/log(n)&=& 0
\end{eqnarray}
The distribution $S\ell(n)/log(n)$ is not normal,
but is symmetric with respect to the value 1,
see \ref{distribution}
\item Lower Envelope:
$\liminf_n S\ell(n) = lim_{n\to\infty} S\ell(q_n) > 0$
\end{enumerate}
Therefore the solution of the linear Master Equation diverges at the radius
of convergence.


\end{theorem}

\vskip.1in
\noindent
\textbf{Proof:} See: \cite{KT}. QED.

\vskip.1in
\noindent
\textbf{Remarks:}
\begin{itemize}
\item The function $\ell(x)$ is not of bounded variation (BV): its derivative has a $\dfrac{1}{x}$ singularity. Therefore Denjoy-Koksma type approaches fail.This is also true for the weaker regularity property proposed
in \cite{HS} (Zygmund and bounded quadratic variation). 
\item This is not surprising since the condition of the problem is for KAM regularity to fail.
\item We have however a new peculiar property: the sequence $S\ell(q_n)$ 
\footnote{We commit the usual transgression in this field: $q_n$ is the denominator
of the $n^{th}$ continued fraction expansion of $\alpha$ and $p_n$ is its numerator.}
is convergent. This could be taken to mean that the quantities
$S\ell(q_n)$ a signed Dirac measure along the orbit is itself
and integral with properties, i.e. there exists a function $\phi$ of class BV and of average equal to zero so that $S\ell(q_n)=\sum_{k=1}^{q_n} \phi(k\alpha)$.

\item
We note that we can reformulate the linear Master Equation in the following
manner. With
$$V(x)=
\left(
\begin{array}{c}
m(x)\\
m(\dfrac{x}{\lambda})
\end{array}
\right)
$$
and 
$$
A(x)\,=\,
\left(
\begin{array}{cc}
2-x & -1 \\
1 & 0 
\end{array}
\right)
$$
The solution of the linear Master Equation is thus of this form:
$$
V(\lambda x)\,=\,A(x)\,V(x)
$$
As a result, when $|x|<1$ and $\lambda$ irrational
the sequence: $A(\lambda^n x)....A(x)$
has bounded solutions states. Surprising considering that
half the time the matrices $A(x)$ are hyperbolic ($trace(A(x))=2-x$).
This fact must have been known by Herman, Yoccoz, etc.
\end{itemize}

\vskip.2in

\begin{figure}[ptbh]
\centering
\scalebox{0.3}{\includegraphics{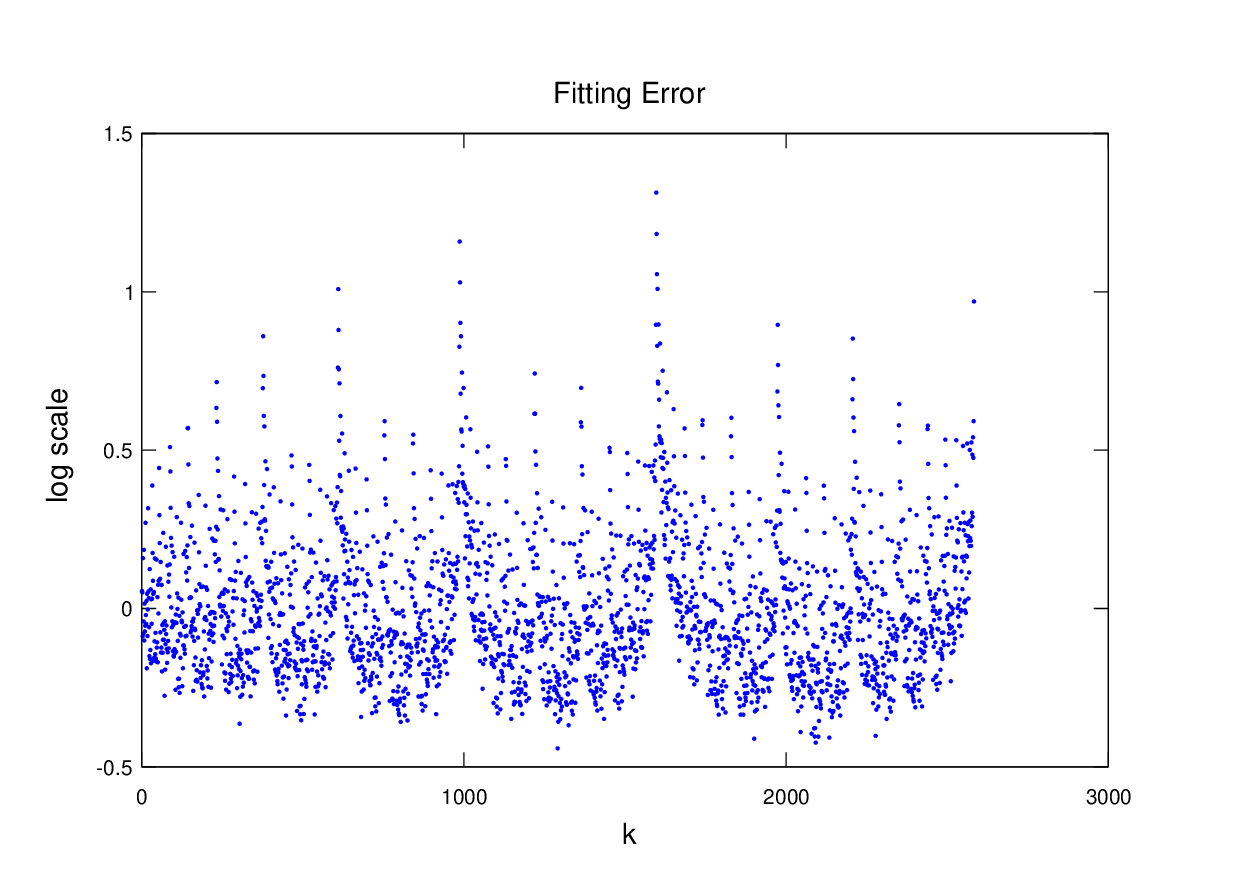}} \\
\caption{\emph{ 
Shown is the result of the difference ('error') between the
values of the sequence $\ln(G(k))$ and the linear fit of
these sequences: $S\ell(k)$, $k$, $1$ and $\ln(k)$ for the same index range. The weights were found to be as follows:
$\ln(G(k))\sim -.910939 S\ell(k) + 0.010279 k -0.208561 - 1.30588 \ln(k)$,
The term linear in $k$ is small and indicates that we are not quite at the radius of convergence, the constant term does not matter, coefficient in front of $\ln(k)$ shows some decay. However what matters is the distribution of $S\ell$. The 'peaks' in the error figure occur at the closest returns $q_n$. This formula 'predicts', since the peak of the distribution for $S\ell(k)/log(k)$ is at 1 that, grosso modo, the terms $\log G(k,r_{\alpha})$ grow as $\sim (-.9 - 1.30)log(k) = (-2.2) \log(k)$, which implies/suggests smoothness on the order of 
$C^{1.2}$.   
 }}
\label{fig:Fitting-Errors-image}
\end{figure}

The goal is to understand the distribution
properties of the sequence $S\ell(n)$. While the mean $\dfrac{S\ell(n)}{n}$ tends to zero, the individual terms are not uncorrelated, and its distribution properties, strangely enough, apparently not broadly understood
\footnote{
There may be work of J-P Conze that is relevant in this connection.
Also there appears to be a connection to Szemeredi theory.}
.

\vskip.2in
\noindent

\begin{centering}\section{Open Problems} \label{section :open problems} \end{centering}
\setcounter{figure}{0} \setcounter{equation}{0}

\vskip .2in
\noindent
\begin{itemize}
\item What of these results extend to the general class
of Bruno numbers?

\item Are there singular examples of the Perez-Marco, Avila-Buff-Cheritat variety of Siegel disks with $C^{\infty}$ boundaries in this example?

\item How do periodic orbits accumulate on Siegel disks? 

\item Consider one of the simplest nonlinear complex dynamical systems in $C^2$:
$$     T(z,w) = (c z , w (1-z) ) \; ,  $$
where $c = \exp(2\pi i \alpha)$ and $\alpha$ is the golden mean. This is one
the simplest quadratic systems, we can write down in $\mathbb{C}^2$. How fast
does the orbit behave on the invariant cylinder $|z|=1 \times \mathbb{C}$? We have
$$     T^n(z,w) = (z_n,w_n) = (c^n z,  w (1-z) (1-c z) 1-c^2 z) ...(1-c^{n-1} z)  $$
and 
$$
\log|w_n| =  \sum_{k=1}^n \log| 1-e^{2\pi i k \alpha} | = S_{n}/2
$$
because $2 \log|1-e^{ix}| =  \log(2-2 \cos(x))$.
It would be nice to study the global behavior of the holomorphic map $T$
in $\mathbb{C}^2$ which implements the random walk over the golden circle on a subset.
\end{itemize}

\section{Appendix} \label{section :Appendix}

\begin{centering}\subsection{Zygmund etc.} \label{section :Zygmund} \end{centering}

We recall the following definitions, see \cite{HS}.
Let $f: \mathbb{S}^1\to \mathbb{R}$ be a function (instead of the circle we could equally consider an interval).

\vskip.1in
\noindent
\begin{definition}
$f$ is Zygmund if there exists $B$ so that
$$
sup_{x,t} \dfrac{|f(x+t)+f(x-t)-2f(x)|}{t}\,\leq\,B
$$
\end{definition}

\vskip.1in
\noindent
\begin{definition}
$f$ is of Bounded Zygmund Variation (BVZ) if there exists $B$ so that
$$
sup_{x_0<x_1<..x_n} \sum_{i=0}^{n-1}
|f(x_i)+f(x_{i+1})-2f(\dfrac{x_i+x_{i+1}}{2})|\,\leq\,B
$$
\end{definition}

\vskip.1in
\noindent
\begin{definition}
$f$ is of Bounded Quadratic Variation (BQZ) if there exists $B$ so that
$$
sup_{x_0<x_1<..x_n} \sum_{i=0}^{n-1}(f(x_i)-f(x_{i+1}))^2\,\leq\,B
$$
\end{definition}

The strongest result to date in Denjoy theory for circle diffeomorphisms is, see \cite{HS}.

\begin{theorem}
A circle diffeomorphism with irrational rotation number
with logarithmic derivative of class BVZ \textbf{and} BQZ is
topologically conjugate to a rotation.
\end{theorem}

\vspace{\fill}

\end{document}